\documentclass{amsproc}
\usepackage{euscript}
\usepackage{cases}
\usepackage{mathrsfs}
\usepackage{bbm}
\usepackage{amssymb}
\usepackage{amsfonts,amsmath,amsxtra,mathdots,mathabx}
\usepackage{color}
\usepackage{hyperref}
\usepackage{tikz}
\usepackage{appendix,upgreek}
\allowdisplaybreaks

\DeclareFontFamily{U}{matha}{\hyphenchar\font45}
\DeclareFontShape{U}{matha}{m}{n}{
	<5> <6> <7> <8> <9> <10> gen * matha
	<10.95> matha10 <12> <14.4> <17.28> <20.74> <24.88> matha12
}{}
\DeclareSymbolFont{matha}{U}{matha}{m}{n}

\DeclareMathSymbol{\Lt}{3}{matha}{"CE}
\DeclareMathSymbol{\Gt}{3}{matha}{"CF}

\DeclareSymbolFont{mathc}{OML}{txmi}{m}{it}
\DeclareMathSymbol{\varuu}{\mathord}{mathc}{117}
\DeclareMathSymbol{\varvv}{\mathord}{mathc}{118}
\DeclareMathSymbol{\varww}{\mathord}{mathc}{119}

\def\valpha{\text{\scalebox{0.84}[1.02]{$\alpha$}}}   
\def\vepsilon{\upvarepsilon}
\def\vnu{\text{{\scalebox{0.9}[1]{$\nu$}}}}

\newcommand{\BR}{{\mathbb {R}}} 
\newcommand{\BZ}{{\mathbb {Z}}}

\newcommand{\SL}{{\mathrm {SL}}}

\newcommand{\ra}{\rightarrow} 
\def\sumx{\sideset{}{^\star}\sum}

\def\nd{\mathrm{d}}

\def\lp {\left (}
\def\rp {\right )}

\def\shskip{\hspace{0.5pt}}

\newcommand{\red}[1]{\textcolor{red}{#1}}

\newcommand{\delete}[1]{}

\theoremstyle{plain}

\newtheorem{thm}{Theorem} \newtheorem{cor}[thm]{Corollary}
\newtheorem{lem}{Lemma}[section] \newtheorem{prop}[thm]{Proposition}

\theoremstyle{remark} 
\newtheorem{remark}{Remark}[section] 

\numberwithin{equation}{section}

\begin{document}
\title[Luo's Spectral Large Sieve Inequality, short-interval]{Luo's  Spectral Large Sieve Inequality on Short Intervals}   

	\author[Z. Qi  and R. Qiao]{Zhi Qi and Ruihua Qiao}
	\address{School of Mathematical Sciences\\ Zhejiang University\\Hangzhou, 310058\\China}
	\email{zhi.qi@zju.edu.cn, ruihua.qiao@zju.edu.cn}
	
	\thanks{The first author was supported by National Key R\&D Program of China (No. 2022YFA1005300) and the National Natural Science Foundation of China (No. 12671014).}
	
	\dedicatory{\normalsize On the occasion of Szu-Hoa Min's 115th birth anniversary}

	\subjclass[2020]{11M41, 11F72}
	\keywords{large sieve inequality, Poisson summation formula, Kuznetsov trace formula.}
	
	\maketitle
\begin{abstract}
Let $u_j $ traverse an orthonormal basis of Hecke--Maass  forms for  $\mathrm{SL}_2 (\BZ) $ with Hecke eigenvalues $\lambda_j (n)$ and Laplace eigenvalue $1/4+t_j^2$. In this paper, we consider the short-interval variant of the twisted spectral large sieve inequality of  Luo  for $ \lambda_j (n) n^{it_j} $ on the range $t_j \leqslant T$ and prove that the  `Eisenstein--Kloosterman' cancellation discovered by Luo is effective on the interval $ T < t_j \leqslant T + M $ as long as $\sqrt{T} < M \leqslant T$.  Moreover, our approach yields an improvement of the large sieve inequality of Luo. 
\end{abstract}

\section{Introduction}
Let $\{u_j (z)\}$ be an orthonormal basis of   Hecke--Maass cusp forms on  $\mathrm{SL}_2 (\BZ) \backslash \mathbb{H}^2$.   Let $\lambda_j = s_j (1-s_j)$  
	be the Laplace eigenvalue of $u_j (z)$,  
	with $s_j = 1/2+ i t_j$ ($t_j > 0$). 
		The Fourier expansion of $u_j (z)$ reads:
	\begin{align*}
		u_j (x+iy) =   \sqrt{y} \sum_{n \neq 0}  \rho_j (n) K_{i t_j} (2\pi |n| y) e (n x), 
	\end{align*}
	where as usual $K_{\vnu} (x)$ is the $K$-Bessel function and $e (x) = \exp (2\pi i x)$. Let $\lambda_j (n)$  be the $n$-th Hecke eigenvalue of $ u_j (z) $. It is well known that $ \rho_j (  n) =   \lambda_j (n) \rho_j (  1)  $ and $ \lambda_j (n) $ is real-valued for any $n \geqslant  1$. 
    Define the harmonic weight 
	\begin{align*}
		\omega_j = \frac {|\rho_j (1)|^2} {\cosh \pi t_j}.
	\end{align*}

    Deshouillers, Iwaniec, and Luo \cite{DI-Nonvanishing,Luo-Twisted-LS} established the following large sieve inequalities for the special twisted Hecke eigenvalues $ \lambda_{j} (n) n^{it_j}  $: 
\begin{align}\label{1eq: DI's bound}
	\sum_{t_j \leqslant T} \omega_j  \bigg| \sum_{ n \leqslant N}  a_{n} \lambda_{j} (n) n^{it_j} \bigg|^2 \Lt \big(T^2 +  N^{2}\big) (TN)^{\vepsilon} \sum_{ n \leqslant N}  |a_{n}|^2 ,  
\end{align}
\begin{align}\label{1eq: Luo's bound, 1}
	\sum_{t_j \leqslant T} \omega_j   \bigg| \sum_{ n \leqslant N}  a_{n} \lambda_{j} (n) n^{it_j} \bigg|^2 \hskip -2pt \Lt    \big(T^2      +       T^{3/2} N^{1/2}     +     N^{5/4} \big) (TN)^{\vepsilon} \hskip -2pt \sum_{ n \leqslant N} \hskip -2pt |a_{n}|^2,
\end{align}
for any complex $a_n$.  
Note that \eqref{1eq: Luo's bound, 1} improves \eqref{1eq: DI's bound} for $N > T$. This improvement is due to Luo's discovery   of the `Eisenstein--Kloosterman' cancellation. 

 It was stated without proof by
 Iwaniec \cite{Iwaniec-Spectral-Weyl} and proven independently by Luo \cite{Luo-LS} and Jutila \cite{Jutila-LS} that 
 \begin{align}\label{1eq: LS, M=1}
 	\sum_{T < t_j \leqslant T+1} \omega_j  \bigg| \sum_{ n \leqslant N}  a_{n} \lambda_{j} (n)  \bigg|^2 \Lt_{\vepsilon}  (T +  N  ) (TN)^{\vepsilon} \sum_{ n \leqslant N}  |a_{n}|^2 ,
 \end{align}
 while Luo observed that, by partial summation, \eqref{1eq: LS, M=1} is equivalent to its twisted variant:
 \begin{align}\label{1eq: LS, M=1, twisted}
 	\sum_{T < t_j \leqslant T+1} \omega_j  \bigg| \sum_{ n \leqslant N}  a_{n} \lambda_{j} (n) n^{it_j}  \bigg|^2 \Lt_{\vepsilon}  (T +  N  ) (TN)^{\vepsilon} \sum_{ n \leqslant N}  |a_{n}|^2 ;  
 \end{align}
 the twist  $n^{it_j}$ does not play a role because $t_j$ is restricted in a segment of unity length. Thus, for $1 \leqslant M \leqslant T$, it follows from \eqref{1eq: LS, M=1, twisted} that 
 \begin{align}\label{1eq: large sieve, short}
 	\sum_{ T < t_j \leqslant T+M }  \omega_j  \bigg| \sum_{     n \leqslant    { N} }  a_{n} \lambda_{j} (n) n^{it_j}  \bigg|^2  \Lt_{\vepsilon}  	M (  T+ N   ) (TN)^{\vepsilon}    \sum_{   n  \leqslant    {N} } |a_{n}|^2 . 
 \end{align}
This was also proven by the first author \cite{Qi-GL(3)-Special-Points} in a direct manner, with no detection of the   `Eisenstein--Kloosterman' cancellation. 

 Presumably, there should still be some effect of cancellation for shorter intervals  $ T < t_j \leqslant T+ M $, but it would disappear on the shortest interval $ T < t_j \leqslant T+ 1 $. 
 
 It is natural to ask for which values of $M$ the `Eisenstein--Kloosterman' cancellation is effective, thereby yielding an improvement on \eqref{1eq: large sieve, short}?    Our main theorem  manifests that such cancellation is effective on short intervals of length $ M > \sqrt{T}$. 
 
 \begin{thm}\label{thm: large sieve}
 	Let   $\sqrt{T} < M \leqslant T< \sqrt{M N}$. We have
 	\begin{align}\label{1eq: large sieve}
 		\sum_{ T< t_j \leqslant T+M }  \omega_j  \bigg| \sum_{     n \leqslant    { N} }  a_{n} \lambda_{j} (n) n^{it_j}  \bigg|^2  \Lt_{\vepsilon}\varLambda(M, T, N) N^{\vepsilon}\sum_{n\leqslant N}|a_n|^2,
 	\end{align}
 	for any complex numbers $a_n$, where 
 {	\begin{align}\label{1eq: Lambda}
 		\varLambda(M, T, N) =\begin{cases}
 			T\sqrt{MN},\  &T^2/ M < N \leqslant M^{3},  \\  
 			TN^{2/3},&M^{3} < N \leqslant T^{3},\\
 			N ,&T^{3} < N . 
 		\end{cases}
 	\end{align}}
 \end{thm}

Note that for $N > T^2/M  $  the bound given by \eqref{1eq: large sieve} and \eqref{1eq: Lambda} is stronger than the bound in \eqref{1eq: large sieve, short}. 

Moreover, if we choose $M = T$, then Luo's large sieve inequality \eqref{1eq: Luo's bound, 1} may be improved as follows.  This is because the use of Taylor approximation as in \cite[\S 3]{Luo-Twisted-LS} is bypassed (see Remark \ref{rem: Taylor}) and the contribution from holomorphic modular forms as in \cite[\S 4]{Luo-Twisted-LS} does not arise in our setting.   

\begin{cor}
	 We have 
	 \begin{align}\label{1eq: Luo's bound, 2}
	 	\sum_{t_j \leqslant T} \omega_j   \bigg| \sum_{ n \leqslant N}  a_{n} \lambda_{j} (n) n^{it_j} \bigg|^2 \hskip -2pt \Lt    \big(T^2      +       T^{3/2} N^{1/2}     +     N  \big) (TN)^{\vepsilon} \hskip -2pt \sum_{ n \leqslant N} \hskip -2pt |a_{n}|^2,
	 \end{align}
 for any complex numbers $a_n$. 
\end{cor}

As remarked in \cite[\S 1.2]{Qi-GL(3)-Special-Points}, if $\phi$ is a fixed Hecke--Maass form for $\mathrm{GL}_3$, then $ M = \sqrt{T} $ is a natural barrier in the subconvexity problem for the `conductor-dropping' Rankin--Selberg special $L$-values $ L (1/2+it_j, \phi \times u_j ) $. However,  our Theorem \ref{thm: large sieve} has just reached this barrier, and it is not yet clear whether the `Eisenstein--Kloosterman' cancellation is still effective beyond $ M  = \sqrt{T} $.

 \subsection*{Setup}

 Let $\mathcal{A}$ be a real sequence $a_n$ with $a_n=0$ unless $N<n\leqslant 2N$ (for the real-valued assumption, the reader is referred to \cite[Remark 1.1]{Qi-GL(3)-Special-Points}); subsequently, we shall write the support as $n \sim N$ and  
 denote 
 \begin{align*}
     \| \mathcal{A} \|^2=\sum_{  n \sim  N} a_n^2.
 \end{align*} 
For the proof of Theorem \ref{thm: large sieve}, we shall deal with the smoothed sum
\begin{align}  \label{1eq: S, cusp}	
	 {S} (\mathcal{A})    =	   \sum_{j = 1 }^{\infty}   \omega_j { h  ( t_j ) }   \bigg| \sum_{ n \sim  N }  a_{n} \lambda_{j} (n) n^{i t_j}  \bigg|^2   ,  
\end{align}
with  
\begin{equation}\label{1eq: weight k} 
	h ( t  ) = \exp \bigg( \hspace{-1pt}  -  \frac {(t - T)^2} {M^2}    \bigg) + \exp \bigg(  \hspace{-1pt} -  \frac {(t + T)^2} {M^2}    \bigg),
\end{equation} 
and the corresponding Eisenstein contribution
\begin{align}	\label{1eq: E, Eis}		  {T} (\mathcal{A})    =	   \frac 1 {\pi} 	\int_{-\infty}^{\infty}   \frac{h (t)}    {|\zeta(1+2it)|^2} \bigg| \sum_{ n \sim  N }   {a}_{n} \sigma_{  2 i t} (n)   \bigg|^2  \nd t ,
\end{align} 
where $\sigma_{\vnu}(n)$ is the divisor function
\begin{align*}
    \sigma_{\vnu}(n)=\sum_{d|n}d^{\vnu}. 
\end{align*}

\subsection*{Strategy for the Proof of Theorem \ref{thm: large sieve}}
Define the bilinear form
\begin{align}\label{1eq: Sigma(A)}
	&\varSigma(\mathcal{A})=M\mathop{\sum\sum}_{m,n} a_m a_n \sigma(m,n), \qquad \sigma(m, n) = \sum_{c=1}^{\infty}   \frac {S(m, 0; c) S(n, 0; c)} {c^2}, 
\end{align}
where
$S(m, 0; c)$  is the Ramanujan sum.

The idea for the proof of Theorem \ref{thm: large sieve} is to extract the same main term—a multiple of $\varSigma (\mathcal{A})$—from the Eisenstein series contribution $T(\mathcal{A})$ and the Kloosterman part in ${S} (\mathcal{A})+T(\mathcal{A})$ after the Kuznetsov formula, and then do the cancellation. More precisely, we shall prove the following asymptotic formulae.

\begin{prop}\label{prop: Kloosterman}
	Let $N^{\vepsilon} \leqslant M \leqslant T^{1-\vepsilon}$ and $ T < \sqrt{MN}$.  Then
	\begin{equation}\label{1eq: Kloosterman}
		{S} (\mathcal{A})+T(\mathcal{A})=\frac{2}{\sqrt{\pi}} \varSigma(\mathcal{A}) + O\bigg(\bigg(     \frac{  \sqrt {N}}{M}  + \sqrt{M }    \bigg) T \sqrt{N}  N^\vepsilon\|\mathcal{A}\|^2\bigg).
	\end{equation}
\end{prop}
 
\begin{prop}\label{prop: Eisenstein}
	Let $N^{\vepsilon} \leqslant M \leqslant T  < N$. 
	Then 
	\begin{equation}\label{1eq: Eisenstein}
		T(\mathcal{A}) = \frac{2}{\sqrt{\pi}} \varSigma(\mathcal{A}) + O\big(N^{1+\vepsilon} \|\mathcal{A}\|^2\big). 
	\end{equation}
\end{prop}

It follows from Propositions  \ref{prop: Kloosterman} and \ref{prop: Eisenstein}   that  
\begin{align}\label{1eq: large sieve, 2}
	\sum_{ T< t_j \leqslant T+M } \! \omega_j  \bigg| \! \sum_{     n \leqslant    { N} } \! a_{n} \lambda_{j} (n) n^{it_j}  \bigg|^2    \Lt_{\vepsilon}  \bigg(  \frac{  \sqrt {N}}{M}  + \sqrt{M }    \bigg) T \sqrt{N}  N^\vepsilon  \sum_{n\leqslant N}|a_n|^2 ,
\end{align}
for   $ 1 <  M \leqslant T < \sqrt{MN}$;  here,  for $T^{1-\vepsilon} < M \leqslant T $, one may always  divide $(T, T+M]$ into $O (T^{\vepsilon})$ many intervals of length $ M / T^{\vepsilon} $. \delete{However, observe that \eqref{1eq: large sieve, 2} is better than \eqref{1eq: large sieve, short} only if
\begin{align}
	T^{4/7} < M \leqslant T, \qquad \frac {T^2} {M} < N < \frac {M^6} {T^2} . 
\end{align}}

Finally, to deduce \eqref{1eq: large sieve} and \eqref{1eq: Lambda} from \eqref{1eq: large sieve, 2}, we use the observation that the spectral sum on the left of \eqref{1eq: large sieve, 2} is non-decreasing in $M$,  so the proof is completed if $M$ is enlarged to   
{\begin{align*}
	  \begin{cases} 
		N^{1/3}, \ & \text{ if } M^{3} < N \leqslant T^{3},\\
		T,&  \text{ if } T^{3} < N   .
	\end{cases}
\end{align*}}

Moreover, we remark that \eqref{1eq: large sieve} and \eqref{1eq: Lambda} may also be deduced from \eqref{1eq: large sieve, short}  in the case $M < N^{\vepsilon}$; hence,   the condition $ N^{\vepsilon} \leqslant M $ is imposed in Propositions \ref{prop: Kloosterman} and \ref{prop: Eisenstein} so that $M$ and $T$ are not excessively small relative to $N$.

\subsection*{Notation}   
By $X \Lt Y$ or $X = O (Y)$ we mean that $|X| \leqslant c Y$  for some constant $c  > 0$, and by $X \asymp Y$ we mean that $X \Lt Y$ and $Y \Lt X$. We write $X \Lt_{\valpha, \beta, ...} Y $ or $  X = O_{\valpha, \beta, ...} (Y) $ if the implied constant $c$ depends on $\valpha$, $\beta$, ....  


By `negligible' or `negligibly small' we mean $ O_A ( T^{-A} )$ for arbitrarily large but fixed $A > 0$. 

Throughout the paper,  $\vepsilon  $ is arbitrarily small and its value  may differ from one occurrence to another.

\section{Preliminaries}

\subsection{Exponential Sums}

Let $e (x) = \exp (2\pi i x)$. 	For integers $  m,  n , q$ and   $c  \geqslant 1$,      define 
\begin{align}\label{1eq: defn Kloosterman}
	S   (m, n ; c ) = \sumx_{   \valpha      (\mathrm{mod} \, c) } e \bigg(   \frac {  \valpha     m +   \widebar{\valpha    } n} {c} \bigg) ,
\end{align} 
\begin{align} \label{2eq: defn V}
	V_{q} (m, n; c) =	\mathop{\sum_{\valpha     (\mathrm{mod}\, c)}}_{ (\valpha     (q-\valpha    ), c) = 1 }  e \bigg(   \frac {  \widebar{\valpha    } m +  \overline{q - \valpha    } n } {c} \bigg),
\end{align}
where the $\star$ indicates the condition $(\valpha    , c) = 1$ and $ \widebar{\valpha    }$ is given by $\valpha     \widebar{\valpha    }  \equiv 1 (\mathrm{mod} \, c)$. The definition of $V_{q} (m, n; c)$ is essentially from Iwaniec--Li  \cite[(2.17)]{Iwaniec-Li-Ortho}. Note that the Kloosterman sum $S (m,n;c)$ is real-valued.

\begin{lem}\label{lem: S = V}
	We have \begin{equation}\label{2eq: S = V}
		S (m, n; c)  e\Big(\frac {m+n} {c} \Big) = \sum_{qr = c } V_q (m, n; r),
	\end{equation}
and
	\begin{equation}\label{2eq: Fourier of V}
	\sum_{\valpha (\mathrm{mod}\,c) }V_\valpha(m,n;c)e\Big(\frac{\valpha q}{c}\Big)=S(m,q;c)S(n,q;c).
\end{equation}
\end{lem}
\begin{proof}
	The first identity is due to Luo \cite[\S 3]{Luo-LS}.  By \eqref{1eq: defn Kloosterman} we write 
	\begin{align*}
		S (m, n; c)  e\Big( \frac {m+n} {c} \Big) = \sumx_{   \valpha      (\mathrm{mod} \, c) } e \bigg(   \frac {  (1-\valpha    ) m +  (1- \widebar{\valpha    }) n } {c} \bigg),
	\end{align*}
	and split the sum according to  $(1-\valpha    , c) = q$. Thus $c = q r$ and $\valpha     = 1 - \widebar{\beta} q$, where $\beta$ ranges over residue classes modulo $r$  such that $  (\beta (q-\beta), r) = 1$. We obtain
	\begin{align*}
		S (m, n; c)  e\Big(\frac {m+n} {c} \Big) =   \sum_{qr=c} \mathop{\sum_{\beta (\mathrm{mod}\, r)}}_{ (\beta (q-\beta), r) = 1 }  e \bigg(   \frac {  \widebar{\beta} m +  \overline{q - \beta} n } {r} \bigg) , 
	\end{align*}
	as desired. 
	
	The second identity follows by a simple calculation:  By  \eqref{2eq: defn V}, its left-hand side equals the double exponential sum 
	\begin{align*}
		\mathop{\mathop{\sum \sum}_{\valpha, \beta (\mathrm{mod}\, c)}}_{ (\beta (\valpha-\beta), c) = (1) }   e \bigg(   \frac {  \widebar{\beta} m +  \overline{\valpha - \beta} n + \valpha q} {c} \bigg)  = {\mathop{\sumx \sumx}_{ \beta, \gamma (\mathrm{mod}\, c)}}  e \bigg(   \frac {  \widebar{\beta} m +  \overline{\gamma} n + (\beta + \gamma) q} {c} \bigg) .
	\end{align*} 
\end{proof}

\subsection{Bilinear Forms with Kloosterman Sums}

For a sequence  $\mathcal{A} = \{ a_{n} \}$, define 
\begin{align*}
	\| \mathcal{A}_{N, \varDelta} \| = \bigg(  \sum_{  {N} < n \leqslant    {N} +  {\varDelta} } |a_{n}|^2 \bigg)^{1/2}, \qquad  
	\| \mathcal{A}_{N} \| = \bigg(  \sum_{ n \sim  N  } |a_{n}|^2 \bigg)^{1/2} .
\end{align*}
\delete{It follows from the Weil bound 
\begin{align*}
	|S (m, n; c)|   \leqslant \tau (c) \sqrt{ (m, n, c) } \sqrt{c} 
\end{align*}  that 
\begin{align}
	\label{2eq: quad form, Kloosterman, x1}  	  \mathop{\mathop{\sum \sum}_{  M < m \leqslant  M + \varDelta }  } _{  N < n \leqslant  N + \varDelta } \!  \big| a_{m}  \overline{b }_{n}   S (m, n; c) \big|   \Lt \tau^2 (c) \sqrt{c} \varDelta \| \mathcal{A}_{M, \varDelta} \| \| \mathcal{B}_{N, \varDelta} \|, 
\end{align} 
and hence 
\begin{align}
	\label{2eq: quad form, Kloosterman, x}  	   \mathop{\sum \sum}_{  m, n  \sim  N  } \big| a_{m}  \overline{a }_{n}   S (m, n; c) \big|   \Lt \tau^2 (c) \sqrt{c} N \| \mathcal{A}_{N}  \|^2, 
\end{align} 
for any complex $a_n$, where  $\tau (c)$ is the number of divisors of $c$. }
By the mean value theorem 
\begin{align*} 
	\sum_{   \valpha      (\mathrm{mod} \, c) }\bigg|\sum_{ N < n\leqslant  N + \varDelta }a_n e\Big(\frac{\valpha n}{c}\Big)\bigg|^2\Lt (c+ \varDelta)\sum_{ N < n\leqslant  N + \varDelta  }|a_n|^2 ,
\end{align*}
along with the Cauchy inequality, we deduce that
\begin{align}\label{2eq: quad form, Kloosterman, 1}
	\mathop{\mathop{\sum \sum}_{  M < m \leqslant  M + \varDelta }  } _{  N < n \leqslant  N + \varDelta } \!  a_{m}  \overline{ a }_{n}   S (m, n; c)    \Lt (c+  {\varDelta}) \| \mathcal{A}_{M, \varDelta} \| \| \mathcal{A}_{N, \varDelta} \|,
\end{align}
and in particular
\begin{align}\label{2eq: quad form, Kloosterman}
	  \mathop{\sum \sum}_{   m, n  \sim  N } a_{m}  \overline{ a }_{n}   S (m, n; c)    \Lt (c+N) \| \mathcal{A}_{N } \|^2 .
\end{align}
Let us also record here a useful inequality for Ramanujan sums:
\begin{equation}\label{eq: quad form,  Ramanujan}
	\sum_{c\leqslant C}\bigg( \sum_{ n \sim  N} | a_nS(n,0;c) | \bigg)^2\Lt CN \log C \sum_{ n \sim  N} | \tau(n) a_n |^2 ,
\end{equation}
which is a direct consequence of $|S(n, 0; c)| \leqslant (n, c)$ and  the Cauchy inequality. 

\subsection{Kuznetsov Trace Formula for {\protect\scalebox{1.06}{$\SL_2 (\BZ)$}}} 

Let  $h (t)$ be an even function satisfying the conditions{\hspace{0.5pt}\rm:}
\begin{enumerate} 
	\item[{\rm (i)\,}] $h (t)$ is holomorphic in  $|\operatorname{Im}(t)|\leqslant {1}/{2}+\vepsilon$,
	\item[{\rm (ii)}] $h(t)\Lt (|t|+1)^{-2-\vepsilon}$ in the above strip. 
\end{enumerate}
Then for $m, n \geqslant   1$ 	we have the following  identity (\cite[Theorem 1]{Kuznetsov}): 
\begin{equation}\label{2eq: Kuznetsov}
	\begin{split}
		\sum_{j = 1}^{\infty}  \omega_j h(t_j)   \lambda_j(m)   \lambda_j(n)  + \frac{1}{\pi} & \int_{-\infty}^{\infty} \omega(t) h(t) (n/m)^{i t} \sigma_{2it}(m)\sigma_{-2it}(n) \nd t\\
		&= \delta_{m, n} \cdot H +   \sum_{c= 1}^{\infty} \frac{S(m, n;c)}{c} H\bigg(\frac{4\pi\sqrt{m n}}{c}\bigg),
	\end{split}
\end{equation}
where  $\delta_{m, n}$ is the Kronecker $\delta$-symbol,  $S (m, n; c)$ is the Kloosterman sum, and 
\begin{align}
	\sigma_{\vnu} (n) = \sum_{d | n} d^{\hspace{0.5pt} \vnu},  
\end{align}
\begin{equation}\label{2eq: omega}
	\omega_j=\frac{ |\rho_j(1)|^2}{\cosh(\pi t_j)}, \qquad \omega (t) = \frac {1} {|\zeta(1+2it)|^2},
\end{equation}
\begin{align}\label{2eq: integrals} 
	H =\frac{1}{\pi^2}\displaystyle\int_{-\infty}^{\infty}h(t)\tanh(\pi t)t \nd t , \qquad  
	H (x)= \frac {2i} {\pi}   \int_{-\infty}^{\infty} J_{2it} (x) h (t) \frac {t \nd t} {\cosh (\pi t )}. 
\end{align}  

\subsection{Hybrid Large Sieve of Young}  
The following hybrid large sieve inequality is a special case of Young's Lemma 6.1 in \cite{Young-GL(3)-Special-Points}, slightly modified in \cite[Lemma 2.6]{Qi-GL(3)-Special-Points}; we shall only need the case $\gamma = 1$ in our later applications.

\begin{lem}\label{lem: Young's LS}
Let $\gamma \neq 0$,  $\tau, v   > 0$,  and $C, N \Gt 1$.  	We have 
	\begin{align}\label{2eq: hybrid ls, Young, 2}
		\begin{aligned}
			\int_{-\tau}^{\tau} \hspace{-1pt} \sum_{c \shskip \leqslant C}  \frac 1 {c} \, \sumx_{   \valpha      (\mathrm{mod} \, c) } \hspace{-1pt} \bigg| \hspace{-1pt}    \sum_{ n  \sim N}    a_{n}   e \Big(   \frac {\valpha     n} {c} \Big)  e  \bigg(   \frac { n^{\gamma} t} {c v} \bigg) \hspace{-1pt} \bigg|^2 \hspace{-1pt}  \nd t \Lt_{\gamma} \big(\tau C + v N^{1-\gamma} \log C \big) \hspace{-2pt}   \sum_{ n \sim N }    |a_n|^2,
		\end{aligned}
	\end{align}
	for any complex $a_n$.  
\end{lem}

\delete{\subsection{Stationary Phase} 
Finally, we record here  \cite[Lemma A.1]{AHLQ-Bessel}, a slightly improved version of  \cite[Lemma {\rm 8.1}]{BKY-Mass}.   
\begin{lem}\label{lem: stationary phase}
	Let $\varww   \in C_c^{\infty} (a, b)$. Let  $f  \in C^{\infty} [a, b]$ be real-valued.  Suppose that there
	are   parameters $P, Q, R, S, Z  > 0$ such that
	\begin{align*}
		f^{(i)} (x) \Lt_{ \, i } Z / Q^{i}, \qquad \varww^{(j)} (x) \Lt_{ \, j } S / P^{j},
	\end{align*}
	for  $i \geqslant 2$ and $j \geqslant 0$, and
	\begin{align*}
		| f' (x) | \Gt R. 
	\end{align*}
	Then 
	\begin{align*}
		\int_a^b  e (f(x)) \varww (x)  \nd x \Lt_{ A} (b - a) S \bigg( \frac {Z} {R^2Q^2} + \frac 1 {R Q} + \frac 1 {R P} \bigg)^A  
	\end{align*} 
	for any  $A \geqslant 0$.
\end{lem}}

\section{Analysis of the Bessel Integral}
 
 Subsequently, we shall always assume $T^{\vepsilon} \leqslant M \leqslant T^{1-\vepsilon}$.  Let us write
 \begin{align}\label{3eq: defn h(t)}
 	h (t) = \upbeta \bigg(     \frac {t - T } {M }    \bigg) \hspace{-1pt} + \upbeta \bigg(     \frac {t + T } {M }    \bigg), \qquad  \upbeta (r) = \exp \big( \!  - r^2\big),
 \end{align}
 and   define
 \begin{equation}\label{3eq: defn h(t; y)}
 	h( t; y ) =  h (t)  	\cos  ( 2  t \log y  ).  
 \end{equation}
 Note that  $h( t; y )$ is even in $t$ as required by Kuznetsov.   
 The purpose of this section is to study its Bessel integral
 \begin{align}
 	H (x; y) = \frac{2 i } {\pi} \int_{-\infty}^{\infty} J_{2it} (x) h (t; y) \frac {t \nd t} {\cosh \pi t }. 
 \end{align}
First,  we record here the bound and  asymptotic formula for $H(x;y)$ in \cite[\S 3]{Qi-GL(3)-Special-Points}. However, instead of the variables introduced therein:
\begin{align*}
	v = \frac {xy} {4}, \qquad w = \frac {x/y} {4}, 
\end{align*}
it will be more convenient in this paper  to set
\begin{equation}
	\label{3eq: v, w} 
	v = \frac {xy+x/y} {4}, \qquad w = \frac {xy-x/y} {4}. 
\end{equation}
\begin{lem}\label{lem: H for small x}
	For  $v  \Lt 1$, we have $H(x;y)=O_{A}(M v /T^{2A})$ for any $A \geqslant 0$. 
\end{lem}
\begin{lem}\label{lem:H-I}
	For $x \Gt 1$, we have the expression
	\begin{align}\label{3eq: H = I}
		\begin{aligned}
			H ( x; y) = 	 \mathrm{Re} \big\{  \exp (2i v) I (v, w)  \big\}  
			+ O_{A}   (T^{-A}   )  ,  
		\end{aligned}
	\end{align}
	for any $A \geqslant 0$, with
	\begin{align}\label{3eq: I}
		I(v, w) =	M T \int_{-M^{\vepsilon}/ M}^{M^{\vepsilon}/M} g (r)  \exp (2 i   (v (\cosh r - 1) + w \sinh r) )  \nd r,
	\end{align}
	in which 
	\begin{align}\label{3eq: defn of beta0 (r)}
		g (r) =	\frac {2 } {{\pi \sqrt{\pi}}} \big(2 \upbeta (Mr) \cos (2Tr) + M/T \cdot \upbeta' (Mr) \sin (2Tr) \big) . 
	\end{align} 
\end{lem}

According to  \eqref{3eq: I} and \eqref{3eq: defn of beta0 (r)}, the integral $ I (v, w) $ may be reformulated as
\begin{align}\label{3eq: split I(v,w)}
	I(v,w) = I_{+} (v, w) + I_{-} (v, w),  
\end{align}
where 
\begin{align}\label{3eq: I+-(v, w)}
	I_{\pm} (v, w) = MT   \int_{-M^{\vepsilon}/ M}^{M^{\vepsilon}/M} g_\pm(r ) \exp(2i \psi(r; v, w)) \exp(2i(w\pm T)r) \nd r, 
\end{align}
\begin{align}\label{3eq: g+-(r)}
	g_\pm(r )=\frac{1}{\pi\sqrt{\pi}} \big( 2 \upbeta(Mr) \mp i{M}/{ T} \cdot \upbeta'(Mr) \big) ,
\end{align} 
\begin{align}\label{3eq: psi(r)}
	\psi(r; v, w)=v(\cosh r-1)+w(\sinh r-r) .
\end{align} 
 
It is proven in \cite[\S 3.3]{Qi-GL(3)-Special-Points} that $ I (v, w)$ is negligibly small if $ v  \Lt T $. The next lemma may be regarded as a refinement of this result.

\begin{lem}\label{lem:I-decay} Let $v > |w| $ and $v \Gt 1$.  Then  \(I_{\pm} (v,w)\) and $I (v, w)$ are negligibly small unless
	\begin{equation}\label{eq:I-effective-support}
		|w|-T 
		\Lt
		T^{\vepsilon}
		\Big(   
		M+\frac{ v }{M}   
		\Big).
	\end{equation}
\end{lem} 

\begin{proof}
On the range \(|r|\leqslant M^\vepsilon/M\), we have
\begin{align*}
	\frac {\partial   \psi(r; v, w)} {\partial r }\Lt \frac{  M^\vepsilon v}{M},
	\qquad
\frac {\partial^{j}  \psi(r; v, w)} {\partial r^j} \Lt_{j} v  , 
	\quad (j\geqslant 2).
\end{align*} 
By the Faà di Bruno formula \cite{Faa-di-Bruno}, we have 
\begin{align*}
	\frac {\partial^{j} \exp(2i \psi(r; v, w))} {\partial r^j} \Lt_j \bigg( \sqrt{v} + \frac {M^{\vepsilon} v} {M}  \bigg)^{j} 
\end{align*}
for any $j \geqslant 0$. 
Hence, if we set $ g_\pm (r; v, w) =  g_\pm(r ) \exp(2i \psi(r; v, w))$,  then it follows from the product rule   that 
\begin{align*}
\frac {\partial^{j} g_\pm  (r; v, w)} {\partial r^j}  \Lt_{j} \bigg(   M + \sqrt{v} + \frac {M^\vepsilon v} {M}   \bigg)^{j} \Lt_j  \bigg(  M +   \frac {M^\vepsilon v} {M}   \bigg)^{j},
\end{align*}
for any $j \geqslant 0$. Thus the Fourier integral in \eqref{3eq: I+-(v, w)} is negligibly small if $|w \pm T | \Gt T^{\vepsilon} (M +v/M) $. 
\end{proof}

\begin{lem}\label{lem:I-main}
	Let $v > |w| $ and $v \Gt 1$. We have 
	\begin{equation}\label{3eq: asymptotic formula of I}
		I(v,w)
		=
		\frac{2 w f (w)}{\pi} 
		+
		O\bigg(\frac{T v }{M^2} \bigg),
	\end{equation}
	where
	\begin{align}\label{3eq: defn of f(w)}
		f (w)=  \upbeta \bigg( \frac{w - T }{M }\bigg) - \upbeta \bigg( \frac{w + T }{M }\bigg), \qquad \upbeta (r) = \exp \big( \!  - r^2\big).
	\end{align}
\end{lem}

\begin{proof}
	For  \( r   \) small, in view of \eqref{3eq: psi(r)}, we have 
	\begin{align*}
		\exp (2i \psi (r)) = 1 + O (   r^2 v ). 
	\end{align*}
Consequently, we deduce from \eqref{3eq: I+-(v, w)} that
\begin{align*}
	I_{\pm} (v, w) = MT   \int_{-\infty}^{\infty}  g_\pm(r )   \exp(2i(w\pm T)r) \nd r + O \bigg(\frac{T v }{M^2} \bigg). 
\end{align*}
By the definition of $g_{\pm} (r)$ in \eqref{3eq: g+-(r)},  the integral may be evaluated explicitly so that
\begin{align*}
I_{\pm} (v, w) =  	\mp \frac {2w} {\pi}   \upbeta \bigg(\frac {w\pm T} {M}\bigg) + O \bigg(\frac{T v }{M^2} \bigg),
\end{align*}
as desired. 
\end{proof}

In practice, we have 
\begin{align*}
	x = \frac {4\pi \sqrt{m n}} {c} , \qquad y = \sqrt{\frac {m} {n} },
\end{align*}
so that
\begin{align*}
	v =  \pi \frac { m +n } {c}, \qquad w =  \pi \frac { m - n } {c}. 
\end{align*}
The last lemma will be useful when we need to separate the variables $m$ and $n$.

\begin{lem}\label{lem: Fourier of f(w)}
	We have
	\begin{align}\label{3eq: f (w) integral}
		f (w) = \frac 2 {i \sqrt{\pi}}  M \! \int_{-\infty}^{\infty} \upbeta (M r) \sin (2T r) \exp (2i r w) \nd r . 
	\end{align}
\end{lem}

\section{Application of the Kuznetsov Trace Formula}
Choose the spectral weight in \eqref{2eq: Kuznetsov} to be $h (t; \sqrt{m/n})$ as defined in \eqref{3eq: defn h(t)} and \eqref{3eq: defn h(t; y)}. Then multiply both sides of \eqref{2eq: Kuznetsov}  by ${a}_{m}   {a}_{n} $, and  sum over $  m, n \sim  {N} $.  Note that $$\mathrm{Re} \big( (m/n)^{i t} \big) = \cos \big(2t \log \sqrt{m/n} \big), $$
so the Kuznetsov formula \eqref{2eq: Kuznetsov}  yields
\begin{align}\label{4eq: C=D+P}
	S (\mathcal{A}) + T (\mathcal{A})  = D (\mathcal{A}) + P (\mathcal{A}), 
\end{align}
with diagonal 
\begin{align}
	D (\mathcal{A}) = H \cdot \sum_{  n } |a_{n}|^2 , \qquad H   = \frac{1}{\pi^2}\displaystyle\int_{-\infty}^{\infty}h(t) \tanh(\pi t)t \nd t , 
\end{align} 
and off-diagonal 
\begin{align}\label{def: off-diagonal}
	P (\mathcal{A}) =  \sum_{c }  	\mathop{\sum\sum}_{m, n  }   {a}_{m}   {a}_{n} \frac{S (m, n;c)}{c} H \bigg( \frac {4\pi \sqrt{m n}} {c} ; \sqrt{\frac {m} {n} } \bigg).
\end{align}
It is clear that 
\begin{align}\label{4eq: bound for diag}
	D (\mathcal{A})\Lt MT \|\mathcal{A}\|^2.
\end{align}
By Lemmas \ref{lem: H for small x}, \ref{lem:H-I}, and \ref{lem:I-decay},  it follows that the Bessel $H$-integral may be transformed into $I$-integral by \eqref{3eq: H = I} while the $c$-sum may be truncated effectively at $ c \asymp N/ T $, so that, up to a negligible error,  $ P (\mathcal{A}) $ is turned into 
\begin{align}\label{4eq: P(a)} 
	\begin{aligned}
	  \mathrm{Re}   \sum_{c \shskip \Lt N/T  }    \mathop{\sum   \sum}_{m, n  \sim  N }    {a}_{m}   {a}_{n}   \frac {S (m, n;c)} { c } e \Big(     \frac {m    +    n} {c}    \Big)  I \Big(  \pi \frac { m +n } {c},  \pi \frac { m - n } {c} \Big) . 
	\end{aligned}
\end{align}     

\section{Reductions for the Application of Poisson}

First, in view of \eqref{3eq: asymptotic formula of I} and \eqref{4eq: P(a)}, we wish to extract from $P (\mathcal{A}) $ the main term: 
	 \begin{align}\label{def: Q(a)}
	Q (\mathcal{A})=2\mathrm{Re}\sum_{c }  	\mathop{\sum\sum}_{m, n  }   {a}_{m}   {a}_{n}(m-n) \frac{S (m, n;c)} {c^2}  e\bigg(\frac{m+n}{c}\bigg)f\bigg(\pi\frac{m-n}{c}\bigg). 
\end{align}
The next lemma is on the estimate for the error term.  

\begin{lem}\label{lem: error estimate}
We have 
\begin{align}\label{5eq: P(A)=Q(A)+O}
	P (\mathcal{A}) = Q (\mathcal{A}) + O\bigg(  \frac{T N }{M } N^{\vepsilon} \|\mathcal{A}\|^2\bigg). 
\end{align}
\end{lem}

\begin{proof}
	 Similar to \eqref{4eq: P(a)}, by the exponential decay of $f (w)$ as in \eqref{3eq: defn of f(w)}, we may restrict the $c$-sum in \eqref{def: Q(a)} as well to the range $ c \Lt N/ T $. For $c \Lt N / T$, let  $E(c; \mathcal{A})$ denote the $c$-th term in the difference $ P (\mathcal{A}) - Q (\mathcal{A}) $. More explicitly, in view of \eqref{3eq: split I(v,w)}, \eqref{3eq: I+-(v, w)}, and the proof of Lemma \ref{lem:I-main}, 
	 \begin{align*}
	 	E (c; \mathcal{A}) = E_+ (c; \mathcal{A}) + E_- (c; \mathcal{A}), 
	 \end{align*}
 where 
 \begin{align*}
 	E_{\pm} (c;\mathcal{A}) = \frac 1 c \mathrm{Re} \mathop{\sum   \sum}_{m, n \sim N }    {a}_{m}   {a}_{n}  S (m, n;c) e \Big(     \frac {m+n} {c}    \Big) K_\pm\Big(  \pi \frac { m +n } {c},  \pi \frac { m - n } {c}\Big),
 \end{align*}
and 
\begin{align*}
	K_\pm(v,w)=MT\int_{-M^{\vepsilon}/M}^{M^{\vepsilon}/M}g_\pm(r) (\exp(2i\psi(r;v,w))-1 ) \cdot 
	\exp(2i(w\pm T)r)
	 \,\nd r.
\end{align*}
Next, we explain how to separate the variables $m$ and $n$ in order to apply \eqref{2eq: quad form, Kloosterman, 1}.  Our idea is to split the difference (here $\psi(r;v,w))$ is defined in \eqref{3eq: psi(r)}) \[  \exp \Big(2i\psi \Big(r;\pi \frac { m +n } {c},  \pi \frac { m - n } {c}\Big) \! \Big)-1 = e \Big( \frac {m + n} {c} (\cosh r  - 1)   +   \frac {m - n} {c} (\sinh r - r) \Big) - 1  \]
into the sum 
\[ a_+ (m/c; r) \cdot (a_-(n/c; r) -  1) + (a_+ (m/c; r) - 1),  \]
with $m$ and $n$ separable in both summands,  for 
\[ a_{\pm} (x; r) =  e  (x (\exp (\pm r) \mp r - 1)  ).   \]
Note that 
\begin{align}\label{5eq: a(x;r) - 1}
	a_{\pm} (x; r) - 1 \Lt \min \big\{ 1, x r^2 \big\}. 
\end{align} 
Now let us proceed and adopt some arguments from \cite[\S 3]{Luo-Twisted-LS}. For  $c \Lt N/T$, define
\delete{\begin{equation*} 
\red{	\varDelta=\min\bigg\{
	N,\,
	T^{\vepsilon} \bigg(c M +\frac{N}{M} \bigg)
	\bigg\}, \quad 
\text{$c T^{\vepsilon} M \Lt N$, so the minimum is just the latter!} }
\end{equation*}}
\begin{equation}\label{5eq: Delta}
	\varDelta = T^{\vepsilon} \max \bigg\{  M c, \frac {  N} {M} \bigg\} ,
\end{equation}
and split the summation over $m, n$ into squares $I \times J$, 
where $I, J$ are sub-intervals of $(N,2N]$ of equal length $\varDelta$. 
For every $I$ there are at most six $J$ whose distances to the shifted intervals $ I \pm c T /\pi $ do not exceed $\varDelta / 2$,   and we only need to consider the contribution from these $I \times J$; otherwise we would have $  | |m - n | - c T /\pi | > \varDelta /2 $ and, in view of Lemma \ref{lem:I-decay} and the choice of $\varDelta$,  the integral $K_{\pm} (\pi (m+n)/ c, \pi (m-n)/ c)$ is negligibly small.


It follows from \eqref{2eq: quad form, Kloosterman, 1},  along with \eqref{5eq: a(x;r) - 1} and \eqref{5eq: Delta},  that
\begin{align*}
	E (c; \mathcal{A}) & \Lt \frac {  T} { c }    (c+ \varDelta) \min \bigg\{ 1, \frac {N} {M^2 c}  \bigg\} N^{\vepsilon}  \|\mathcal{A}\|^2 \\
	& \Lt \frac {  T} { c }  \max \bigg\{  M c, \frac { N} {M} \bigg\}  \min \bigg\{ 1, \frac {N} {M^2 c}  \bigg\} N^{\vepsilon}  \|\mathcal{A}\|^2\\
	& =  \frac {   T N } { M c }    N^{\vepsilon}  \|\mathcal{A}\|^2 . 
\end{align*}  
Note that the sum of $ \|\mathcal{A}_{N_I, \varDelta } \|  \cdot        \|\mathcal{A}_{N_J, \varDelta } \|  $ (say $I = (N_I, N_I +\varDelta  ]$) for those $I \times J$ under concern is bounded by $ 6 \| \mathcal{A}_{N} \|^2 $. Finally, by summing  the estimate above for $ E (c; \mathcal{A}) $ over $c $, we obtain the error bound in \eqref{5eq: P(A)=Q(A)+O}.  
\end{proof}

\begin{remark}\label{rem: Taylor}
	 The results in an earlier draft of the paper is weaker, as we followed Luo \cite[\S 3]{Luo-Twisted-LS} and used   Taylor expansions to keep $m$ and $n$ separate: 
	 \begin{align*}
	 	e \Big(      \frac {m + n} {c} (\cosh r  - 1)  \Big)=   1 +  \sum_{k=1}^{\infty}  \frac {  (2\pi i)^k } {k!} \Big( \frac {m + n} {c} (\cosh r  - 1) \Big)^k , \\
	 	e \Big(      \frac {m - n} {c} (\sinh r - r)  \Big)=   1 +  \sum_{k=1}^{\infty}  \frac {  (2\pi i)^k } {k!} \Big( \frac {m - n} {c} (\sinh r - r) \Big)^k. 
	 \end{align*}
	 Since $ \cosh r  - 1  = O (r^2)$ and $\sinh r - r = O (r^3)$, 
	 this approximation argument  (only) works effectively  when 
	 \begin{align*}
	 	c > \frac {M^{\vepsilon} N  }   {M^{2} } , 
	 \end{align*} so that the Taylor series are rapidly convergent. 
	 However, in this way, we have to apply the Weil bound in the trivial manner if $c \leqslant M^{\vepsilon} N / M^2 $, yielding the weaker estimate:  \begin{align*}
	 	E (c; \mathcal{A}) \Lt   \frac {T N } {M \sqrt{c}}  N^{\vepsilon}  \|\mathcal{A}\|^2.
	 \end{align*}  
\end{remark}

By applying the Luo  identity \eqref{2eq: S = V} in Lemma  \ref{lem: S = V}, we may rewrite \eqref{def: Q(a)} as 
\begin{align}\label{def: Q(a) open S to V}
	Q (\mathcal{A})=2\mathrm{Re} \mathop{\sum\sum}_{c,q }  \frac{1} {c^2q^2}	\mathop{\sum\sum}_{m, n  }   {a}_{m}   {a}_{n}   (m-n)V_q(m,n,c)f\bigg(\pi\frac{m-n}{cq}\bigg). 
\end{align}
Finally, let $1 < X \Lt \min \{ M, N/ T  \}$ be a parameter to be chosen optimally later. Let $Q(\mathcal{A};X)$
be the partial sum of \eqref{def: Q(a) open S to V} restricted by $c\leqslant X$. It follows from the hybrid large sieve inequality \eqref{2eq: hybrid ls, Young, 2} in Lemma \ref{lem: Young's LS} that 
\begin{align}\label{6eq: Q=QX+error}
	Q (\mathcal{A})=Q (\mathcal{A};X)+O\bigg( \frac{M N}{X} N^{ \vepsilon}\|\mathcal{A}\|^2\bigg).
\end{align}
To see this, for dyadic $X \Lt C \Lt N / T$, we denote by $Q_{C} (\mathcal{A})$ the partial sum over $C < c \leqslant 2 C$.  Then we separate $m$ and $n$ by the integral representation of $f (w)$  in \eqref{3eq: f (w) integral}, open \(V_q(m,n;c)\)  by its definition in \eqref{2eq: defn V}, and apply Lemma \ref{lem: Young's LS} with $\tau = M^{\vepsilon} / M$ and $v =  q$ to deduce
\begin{align*}
	Q_C(\mathcal{A})
	\Lt
	\frac{MN}{C} \!\!
	\sum_{q\, \Lt N/ T} \frac1{q^2} \bigg(
	\frac{C }{M} + q \bigg) N^{\vepsilon} 
	\|\mathcal{A}\|^2                                
	\Lt \! \bigg( N+\frac{MN}{C}	\bigg)	N^\vepsilon
	\|\mathcal{A}\|^2 
	\Lt  \frac{MN}{X} 
	N^\vepsilon
	\|\mathcal{A}\|^2.
\end{align*}

\section{Application of the Poisson Summation Formula}
Let us rearrange the sum $Q(\mathcal{A};X)$ as follows
\begin{align}
	Q (\mathcal{A};X)=2\mathrm{Re}\sum_{c\leqslant X }  \frac{1} {c^2}	\mathop{\sum\sum}_{m , n  }   {a}_{m}   {a}_{n}   (m-n)  \sum_{q}\frac{V_q(m,n;c)}{q^2}f\bigg(\pi\frac{m-n}{cq}\bigg). 
\end{align}
As in \cite{Iwaniec-Li-Ortho,Young-GL(3)-Special-Points}, we introduce a non-decreasing weight function $\eta \in C^{\infty} (\BR)$ such that $ \eta (x) \equiv 0$ on $(-\infty, 1/2]$ and  $ \eta (x) \equiv 1$ on $(1, \infty]$. Now  the $q$-sum  may be rewritten as
\begin{align*}
	 \sum_{q} V_q(m,n;c) \frac{\eta(q)}{q^2}
	f\bigg(\pi\frac{m-n}{cq}\bigg).
\end{align*}
By an application of the Poisson summation formula modulo $c$, this is  transformed into
\begin{align*}
	\frac{1}{c}\sum_{\valpha (\mathrm{mod}\,c) }V_\valpha(m,n;c)\sum_{q}e\Big(\frac{\valpha q}{c}\Big)\int \eta(x)f\Big(\pi\frac{m-n}{cx}\Big)e\Big(\! -\frac{x q}{c}\Big)\frac{\nd x}{x^2}.
\end{align*} 
Recall from  \eqref{2eq: Fourier of V} that 
\begin{equation*}
	\sum_{\valpha (\mathrm{mod}\,c) }V_\valpha(m,n;c)e\Big(\frac{\valpha q}{c}\Big)=S(m,q;c)S(n,q;c). 
\end{equation*}
Let us also introduce the Fourier integral 
\begin{align}\label{defn: Fourier integral}
	\phi (w;v)=w\int \eta(x)f(w/x)e(-v x)\frac{\nd x}{x^2}. 
\end{align}

Consequently, we obtain after the Poisson summation: 
\begin{align}\label{6eq: Q=dual sum+zero frequency}
	Q (\mathcal{A};X)=S (\mathcal{A};X)+Z (\mathcal{A};X), 
\end{align}
where $Z (\mathcal{A};X)$ is the zero frequency
\begin{align}\label{6eq: Z(A;X)}
	Z (\mathcal{A};X)=\frac{2}{\pi}\sum_{c\leqslant X }  \frac{1} {c^2}	\mathop{\sum\sum}_{m , n  }   {a}_{m}   {a}_{n}   S(m,0;c)S(n,0;c)\phi \bigg(  \pi\frac{m-n}{c};0   \bigg),
\end{align}
and $S (\mathcal{A};X)$ is the dual sum
\begin{align}\label{6eq: S(A;X)}
	S (\mathcal{A};X)=\frac{2}{\pi} \mathrm{Re} \sum_{c\leqslant X }  \frac{1} {c^2}	\mathop{\sum\sum}_{m , n  }   {a}_{m}   {a}_{n} \sum_{q\neq 0}  S(m,q;c)S(n,q;c)\phi \bigg(\pi\frac{m-n}{c};\frac{q}{c}\bigg).
\end{align}

\begin{prop}\label{prop: zero frequency}
	Let $T < N$ and $1 < X \Lt \min \{ M, N/T \}$. Then
	\begin{align}\label{6eq: asym for Z}
		Z(\mathcal{A};X)=\frac{2}{\sqrt{\pi}}\varSigma(\mathcal{A})+O\bigg(\bigg(\frac{N}{X}+T\bigg)MN^\vepsilon\|\mathcal{A}\|^2\bigg).
	\end{align}
\end{prop}
\begin{prop}\label{prop: dual sum}
	Let $T < N$ and $1 < X \Lt \min \{ M, N/T \}$. Then
	\begin{align}\label{eq: asymptotic for S}
		S(\mathcal{A};X)=S_0(\mathcal{A};X)+O(\|\mathcal{A}\|^2),
	\end{align}
	where 
	\begin{align}\label{eq: bound for S0}
		S_0(\mathcal{A};X)\Lt T^2N^{\vepsilon} \! \sum_{c\leqslant X}\frac{1}{c}\sum_{0<|q|\leqslant c  T N^\vepsilon /M}\frac{1}{|q|} \! \int_{- N^\vepsilon/M}^{N^\vepsilon/M}\bigg|\sum_n a_n S(n,q;c)e\Big(\frac{n r}{c}\Big)\bigg|^2\nd r,
	\end{align}
	and hence 
	\begin{align}\label{6eq: bound for dual sum}
		S(\mathcal{A};X)\Lt T^2XN^\vepsilon \|\mathcal{A}\|^2.
	\end{align}
\end{prop}

\section{Analysis of the Fourier Integral}

\begin{lem}\label{lem: bound for f}
	 $ \phi (w; v) $ is negligibly small unless $ |w| \Gt T$ and $ v \Lt   T^{1+\vepsilon}   / M $. 
\end{lem}

\begin{proof}
By the change $ x \ra |w| x / T $, it follows from \eqref{3eq: defn of f(w)} and \eqref{defn: Fourier integral} that 
	\begin{align*} 
	\phi (w; v) = T \int_{\frac T {2 | w|} }^{\infty}   \eta \bigg(\frac {|w| x } {T} \bigg) f \bigg(\frac {T} {x} \bigg) e \bigg( \! - \frac {  |w | v x } {T}  \bigg) \frac {\nd x} {x^2} ,
	\end{align*}
with 
\begin{align*}
f \bigg(\frac {T} {x} \bigg) =	\upbeta \bigg( \frac T M  \bigg(1-\frac 1 x   \bigg)  \bigg) - \upbeta \bigg( \frac T M  \bigg(1+\frac 1 x  \bigg)   \bigg) , \qquad \upbeta (r) = \exp \big( \!  - r^2\big) . 
\end{align*}
It is clear that 
$f (T/x)$ and its derivatives are   exponentially small   outside the range $ |x-1| < T^{\vepsilon}M/ T$. Thus   the integral is negligibly small if  $ | w | \Lt T $.  Moreover, for $x$ close to $1$,  we have  
\begin{align*}
	\frac {\nd^j f(T/x)} {\nd x^j} \Lt_j \lp \frac T M \rp^{j}, 
\end{align*}
and hence
\begin{align*}
	\frac {\nd^j ( \eta  (  {|w| x } / {T}  ) f(T/x))} {\nd x^j} \Lt_j \lp \frac {|w|} {T} +  \frac T M \rp^{j}. 
\end{align*} 
It follows that the Fourier integral above is negligibly small unless   \begin{align*}
	\frac{|w v |  }{T}   \Lt T^{\vepsilon} \lp \frac {|w|} {T} +  \frac T M \rp, 
\end{align*} or, in the case that $|w| \Gt T$,  unless
\begin{align*}
	 |v| 
	 \Lt   \frac {T^{1+\vepsilon} } { M}  . 
\end{align*}    
\end{proof}

Next, for $v \neq 0$, we consider the Fourier transform
\begin{align*}
	\hat{\phi}(r;v)=\int \phi(w;v) e(-rw)\nd w.
\end{align*}

\begin{lem}\label{lem: Bound for Fourier transform}
Define
	\begin{equation}\label{7eq: k(r)}
		k (r)= \frac 1 {\sqrt{\pi}} M\upbeta (Mr) \sin (2Tr) .
	\end{equation}
 	We have 
	\begin{align}\label{7eq: hat phi}
		\hat{\phi}(r;v)= \pi \int \eta(x)k' (\pi r x)e(-v x) {\nd x} ,
	\end{align}
	and consequently the uniform bounds 
	\begin{align}\label{eq: bounds for phi(r)}
		\hat{\phi} \Big(\frac{r}{\pi};v  \Big)\Lt \frac{T}{|r|}, \qquad 
		\hat{\phi} \Big(\frac{r}{\pi};v \Big)\Lt \frac{T^2}{|v|},
	\end{align}
	and for $|r|>1/M$,
	\begin{align}\label{eq: bound for phi(r), 2}
		\hat{\phi} \Big(\frac{r}{\pi};v  \Big) \Lt MT \, \exp \bigg( \! - \frac{M^2 r^2}{ 4 } \bigg).
	\end{align}
\end{lem}

\begin{proof}
	By definition, we have 
	\begin{align*}
		\hat{\phi}(r;v)=&\int e(- rw) w \bigg(\int   \eta(x)f(w/x)e(-v x)\frac{\nd x}{x^2}\bigg)\nd w\\
		=& -\frac{1}{2\pi i}\frac{\partial}{\partial r}\iint \eta(x)f(w /x) e(- rw-v x) \frac{\nd x \nd w }{x^2}.
	\end{align*} 
After  reversing the order of integrations,   we obtain \eqref{7eq: hat phi} by a direct evaluation of the $w$-integral using the Fourier inversion of \eqref{3eq: f (w) integral} in Lemma \ref{lem: Fourier of f(w)}. However, the double integral is not absolutely
integrable. This issue may be easily addressed by replacing $x^2$ by $x^{2+\vepsilon}$ 
in the denominator and letting  $\vepsilon \rightarrow 0^+$ at the end. 

Since  
\begin{align*}
	k' (r) \Lt ( M T + M^3 |r|) \upbeta (M r), 
\end{align*}
by trivial estimation, it follows from \eqref{7eq: hat phi} that 
\begin{align*}
		\hat{\phi}(r/\pi ; v) & \Lt \int_{   1 / 2}^{\infty} \big(MT + M^3 |rx| \big)  \upbeta (M r x) \nd x  = \frac 1 {|r |}  \int_{     \frac {M |r|}  2}^{\infty}  (T + M  x  )  \upbeta ( x) \nd x, 
\end{align*}
and hence the first bound in \eqref{eq: bounds for phi(r)} and \eqref{eq: bound for phi(r), 2}; for the former, we just extend the domain to $(0, \infty)$.   
By applying partial integration to  \eqref{7eq: hat phi}, we have
\begin{align*}
	\hat{\phi}(r;v) = \frac 1 {2 i v} \int \big(\eta'(x)k' (\pi r x) + \pi r \eta (x) k''(\pi r x) \big)e(-v x) {\nd x}. 
\end{align*}
Note that 
\begin{align*}
	k'' (r) \Lt  ( M T^2 + M^3T |r| + M^5 r^2  ) \upbeta (M r), 
\end{align*}
so we may prove the second bound in  \eqref{eq: bounds for phi(r)} as well by trivial estimation. 
\end{proof}

\section{Proof of Propositions \ref{prop: zero frequency} and \ref{prop: dual sum}}

\subsection{Treatment of the Zero Frequency}\label{subsec: Zero Frequency}

Let $\delta   (x)=\eta(1/x)-1$. Note that $\delta   '(x)$ is supported on $[1,2]$. By \eqref{defn: Fourier integral}, for $w \neq 0$, we have   
\begin{align*}
	\phi(w;0)=w \int_0^\infty \eta(1/x)f(w x)\nd x=\int_0^\infty f(x)\nd x+w \int_0^\infty    \delta   (x)f(w x)\nd x.
\end{align*}
Recall the definition from \eqref{3eq: defn of f(w)}:
\begin{align*}
	f (x)=  \upbeta \bigg( \frac{x - T }{M }\bigg) - \upbeta \bigg( \frac{x + T }{M }\bigg), \qquad \upbeta (r) = \exp \big( \!  - r^2\big). 
\end{align*} It is clear that 
\begin{align*}
	\int_0^\infty f(x)\nd x  = \! \int_{-\infty}^{\infty} \! \upbeta \bigg(\frac {x-T} {M} \bigg) \nd x + O \bigg(\frac{M^2}{T}\upbeta \bigg( \frac{T }{M }\bigg)\bigg)  = \sqrt{\pi} M +  O\bigg(\frac{M^2}{T}\upbeta \bigg( \frac{T }{M }\bigg)\bigg) ;
\end{align*}
 the error is negligibly small for $M \leqslant T^{1-\vepsilon}$. Therefore, in view of \eqref{1eq: Sigma(A)} and \eqref{6eq: Z(A;X)}, the first integral contributes 
 \begin{align*}
 	\frac{2}{\sqrt{\pi}}M\sum_{c\leqslant X}\frac{1}{c^2}\bigg(\sum_n a_nS(n,0;c)\bigg)^2 + O \big(\|\mathcal{A}\|^2\big) =\frac{2}{\sqrt{\pi}}\varSigma(\mathcal{A})+O\bigg(\frac{MN^{1+\vepsilon}}{X}\|\mathcal{A}\|^2\bigg),
 \end{align*}
 where the tail of the sum has been estimated trivially by \eqref{eq: quad form,  Ramanujan}. As for the second integral, in view of \eqref{3eq: f (w) integral} in Lemma \ref{lem: Fourier of f(w)}, we may write  
 \begin{align*}
 	f(wx)= - 2 i  \int_{-\infty}^{\infty} k (r)  \exp (2i r w x) \nd r , \qquad k (r)= \frac 1 {\sqrt{\pi}} M\upbeta (Mr) \sin (2Tr). 
 \end{align*}
Note that $k (r)$ was also defined in \eqref{7eq: k(r)}. Consequently, by this expression and partial integration, the second integral reads
\begin{align*}
	w \int_0^\infty   \delta   (x)f(w x)\nd x & = -2 i   w   \int_0^\infty \delta   (x) \int_{-\infty}^{\infty} k (r) e(rwx/\pi ) \nd r\, \nd x \\
	& =   \int_1^2 \delta' (x) \int_{-\infty}^{\infty} \frac {k (r)} {r} e(rwx/\pi )    {\nd r}   \, \nd x . 
\end{align*} 
Thus one is reduced to estimating 
\begin{align*}
	\int_{-\infty}^{\infty}
	\frac{k(r)}{ r} 
	\sum_{c\leqslant X}\frac{1}{c^2}
	\mathop{\sum \sum}_{m , n}
	a_m a_n S(m,0;c)S(n,0;c)
	e\bigg(\frac{r x (m-n)}{c}\bigg)
	\nd r,
\end{align*}
for any $1\leqslant x \leqslant 2$. By the definition of $k (r)$, up to a negligible error,  this is bounded by 
\begin{align*}
	MT\int_{-M^\vepsilon/M}^{M^\vepsilon/M}
	\sum_{c\leqslant X}\frac{1}{c^2}
	\bigg|
	\sum_n a_n S(n,0;c)e\Big(\frac{r xn}{c}\Big)
	\bigg|^2
	  {\nd r}   .
\end{align*}
Further, open the Ramanujan sum and apply Cauchy–Schwarz to bound this by 
\begin{align*}
	MT\int_{-M^\vepsilon/M}^{M^\vepsilon/M}
	\sum_{c\leqslant X}\frac{1}{c} \  \sumx_{   \valpha      (\mathrm{mod} \, c) }
	\bigg|
	\sum_n a_n e\Big(\frac{\valpha n}{c}\Big)e\Big(\frac{r xn}{c}\Big)
	\bigg|^2
	\nd r. 
\end{align*}
Finally, the hybrid large sieve of Young in \eqref{2eq: hybrid ls, Young, 2}, with $\tau = M^\vepsilon/M$, $v = 1/x$, and $C = X$, 
yields the  bound $O(M T(1+X/M) N^{\vepsilon}\|\mathcal{A}\|^2)=O(M T N^{\vepsilon} \|\mathcal{A}\|^2)$. This completes the proof of Proposition \ref{prop: zero frequency}.

\subsection{Treatment of the Dual Sum}
Now we consider the dual sum $S (\boldsymbol{a}; X) $ as given in \eqref{6eq: S(A;X)} and prove Proposition \ref{prop: dual sum}. 

First of all,   Lemma \ref{lem: bound for f} suggests that we may assume in practice that $0<|q|\leqslant cTN^\vepsilon/M$ as in our case $v=q/c$. By Fourier inversion,
\begin{align*}
	 \phi \bigg(\pi\frac{m-n}{c};\frac{q}{c}\bigg)=\frac 1 {\pi} \int_{-\infty}^{\infty}  \hat{\phi}\bigg(\frac{r}{\pi};\frac{q}{c}\bigg) e \Big(\frac{m-n}{c} r \Big)\nd r.
\end{align*}
The bound \eqref{eq: bound for phi(r), 2} for $\hat{\phi}(r/\pi; v)$ in Lemma \ref{lem: Bound for Fourier transform} suggests that we may truncate the integral at $|r| = N^\vepsilon/M$. Therefore, up to a negligible error, $S(\mathcal{A}; X)$ can be rewritten as 
\begin{align*}
\frac 2 {\pi^2} \mathrm{Re}	\sum_{c\leqslant X }  \frac{1} {c^2}	 \sum_{0<|q|\leqslant cTN^\vepsilon\! /M}  \int_{-N^\vepsilon/M}^{N^\vepsilon/M} \hat{\phi}\bigg(\frac{r}{\pi};\frac{q}{c}\bigg)\bigg|\sum_n a_nS(n,q;c)e\Big(\frac{n r}{c}\Big)\bigg|^2 \nd r. 
\end{align*}
Now \eqref{eq: asymptotic for S} and \eqref{eq: bound for S0} are  direct consequences of the second bound in  \eqref{eq: bounds for phi(r)} for $\hat{\phi}(r/\pi; v)$ in Lemma \ref{lem: Bound for Fourier transform}.

Finally, we need to estimate 
\begin{align*}
	T^2N^{\vepsilon} \! \sum_{c\leqslant X}\frac{1}{c}\sum_{0<|q|\leqslant c  T N^\vepsilon /M}\frac{1}{|q|} \! \int_{- N^\vepsilon/M}^{N^\vepsilon/M}\bigg|\sum_n a_n S(n,q;c)e\Big(\frac{n r}{c}\Big)\bigg|^2\nd r,
\end{align*}
and prove that it has the bound as in \eqref{6eq: bound for dual sum}. To this end, 
extend the sum over $q$ to $O( T N^{\vepsilon}/ M)$ many complete sums modulo $c$, bound $1/|q|$ by either $1 $ or $1/   \lfloor (|q|-1)/ c \rfloor $ according as $|q| \leqslant c$ or not,  open the square and the Kloosterman sums, and execute the summation over $q$. Then we arrive at the expression
\begin{align*}
	T^2 N^\vepsilon   \sum_{c \leqslant X}    \int_{- N^\vepsilon/ M}^{N^\vepsilon/ M} \, \sumx_{\valpha (\mathrm{mod} \, c) }\bigg|\sum_n a_n e\Big(\frac{\valpha n}{c}\Big) e \Big(\frac{n r}{c}\Big)\bigg|^2\nd r, 
\end{align*}
and an application of the hybrid large sieve of Young in \eqref{2eq: hybrid ls, Young, 2} yields the desired bound  $   O (T^2X N^\vepsilon \|\mathcal{A}\|^2 ) $.

\section{Proof of Proposition \ref{prop: Kloosterman}}

On the identities or asymptotics in \eqref{4eq: C=D+P}, \eqref{4eq: bound for diag}, \eqref{5eq: P(A)=Q(A)+O}, \eqref{6eq: Q=QX+error}, \eqref{6eq: Q=dual sum+zero frequency}, \eqref{6eq: asym for Z}, and \eqref{6eq: bound for dual sum}, we have established 
\begin{align*}
	{S} (\mathcal{A}) +{T} (\mathcal{A})=\frac{2}{\sqrt{\pi}}\varSigma(\mathcal{A}) + O\bigg(    \frac{T N }{M } +\frac{MN}{X}+T^2X\bigg)N^\vepsilon\|\mathcal{A}\|^2\bigg),
\end{align*}
for any  $1 < X \Lt \min \{ M, N/T \}$. Thus, for $T < \sqrt{MN}$,  we obtain \eqref{1eq: Kloosterman} by choosing $ X = \min \big\{ M,   \sqrt{MN}/ T \big\} $.

 \section{Proof of Proposition \ref{prop: Eisenstein}} 

In this section, we investigate the Eisenstein contribution ${T} (\mathcal{A})$ as defined in \eqref{1eq: E, Eis} and prove
its asymptotic formula in \eqref{1eq: Eisenstein}.

We start with the Ramanujan identity 
\begin{equation*}
\frac{\sigma_{1-s} (n) }{\zeta(s)} =	\sum_{c} \frac {S(n, 0;c)} {c^s} , \qquad \text{($\mathrm{Re}(s) > 1$)} .
\end{equation*}
For $\mathrm{Re} (s)=1$, by the same argument of \S 3.13 by Lemma 3.12 in \cite{Titchmarsh-Riemann}, it follows that  
\begin{equation}\label{eq: c-truncate}
	\frac{\sigma_{1-s} (n) }{\zeta(s)} =\sum_{\log c\, \leqslant Y^\vepsilon} \frac{S(n,0;c)}{c^{s}} +O(n Y^{\vepsilon}/ Y),
\end{equation}
for any $\mathrm{Im} (s) \Lt Y$. 

Next, we expand and rewrite \eqref{1eq: E, Eis}  as
\begin{equation*}
	{T} (\mathcal{A}) = \frac{1}{\pi}\mathop{\sum\sum}_{m,n} a_m a_n \int h(t) \frac{\sigma_{2it}(m)\sigma_{-2it}(n)}{|\zeta(1+2it)|^2}\nd t.
\end{equation*}
Now let $s = 1 \pm 2 i t$,  $Y = M N $ in  \eqref{eq: c-truncate}, then we have
\begin{align}
	{T} (\mathcal{A}) = \frac{1}{\pi}\mathop{\sum\sum}_{m,n} a_m a_n X(m,n)+O\big(  N^{1+\vepsilon} \|\mathcal{A}\|^2\big), 
\end{align}
with 
\begin{align}\label{10eq: X(m,n)}
	X(m,n)=\mathop{{\sum\sum} }_{\log b,\log c \, \leqslant N^\vepsilon}\frac{S(m,0;b)S(n,0;c)}{bc}k^{}_{^\natural} (\log(b/c)),
\end{align}
\begin{align}\label{eq:kn (r)}
	k^{}_{^\natural} (r)
	=
	\int h(t)\exp(2it r)\,\nd t = 2 \sqrt{\pi} M \upbeta (M r) \cos (2T r);
\end{align}
in particular,
\begin{equation}
	k^{}_{^\natural} (0)=2\sqrt{\pi}M.
\end{equation}

According as \(b=c\) or \(b\ne c\) in the sum $X(m,n)$ in \eqref{10eq: X(m,n)}, we split
\begin{align}\label{10eq: T=T0+Tn}
	{T} (\mathcal{A}) = {T}_0 (\mathcal{A}) + {T}_{\natural} (\mathcal{A}) +O\big(  N^{1+\vepsilon} \|\mathcal{A}\|^2\big). 
\end{align}
We have
\begin{align*}
	T_0(\mathcal{A})=\frac{2 }{\sqrt{\pi}} M \mathop{\sum\sum}_{m,n} a_m a_n \sum_{\log c\, \leqslant N^\vepsilon} \frac{S(m,0;c)S(n,0;c)}{c^2}, 
\end{align*}
and if we estimate the tail of the $c$-sum by \eqref{eq: quad form,  Ramanujan}, then  
\begin{equation}\label{10eq: T0(A)}
	T_0(\mathcal{A})=\frac{2}{\sqrt{\pi}} \varSigma(\mathcal{A})+O\big(M\|\mathcal{A}\|^2\big).
\end{equation}
We have
\begin{equation*}
	T_{\natural}(\mathcal{A})\Lt M\mathop{\mathop{\mathop{\sum\sum}_{b\neq c}}_{\log b,\log c\, \leqslant N^\vepsilon}}_{|\log(b/c)|\leqslant N^{\vepsilon}/M}\mathop{\sum\sum}_{m,n}|a_ma_n|\frac{|S(m,0;b)S(n,0;c)|}{bc},
\end{equation*}
due to the exponential decay of $ k^{}_{^\natural} (r) $ as in \eqref{eq:kn (r)}. By the AM–GM inequality and by symmetry, the sum above is bounded by
\begin{align*}
	M \! \sum_{\log c\, \leqslant N^\vepsilon} \!   \frac{1}{c^2} \!\! \mathop{\sum_{b\neq c}}_{|\log(b/c)|\leqslant N^{\vepsilon}/M} \! \! \! \bigg( \sum_n |a_n S(n,0;c) | \bigg)^2 \! \Lt N^{\vepsilon} \! \sum_{\log c\, \leqslant N^\vepsilon} \! \frac{1}{c } \bigg( \sum_n |a_n S(n,0;c) | \bigg)^2.
\end{align*}
Thus an application of \eqref{eq: quad form,  Ramanujan} yields 
\begin{equation}\label{10eq: Tn(A)}
	T_{\natural}(\mathcal{A})\Lt N^{1+\vepsilon}\|\mathcal{A}\|^2.
\end{equation}
Finally, we conclude the proof of Proposition \ref{prop: Eisenstein} by \eqref{10eq: T=T0+Tn}--\eqref{10eq: Tn(A)}. 


\begin{thebibliography}{Luo2}
	
	\bibitem[DI]{DI-Nonvanishing}
J.-M. Deshouillers and H.~Iwaniec.
\newblock The nonvanishing of {R}ankin-{S}elberg zeta-functions at special
points.
\newblock   {\em The {S}elberg {T}race {F}ormula and {R}elated {T}opics
	({B}runswick, {M}aine, 1984)},  {Contemp. Math., vol. 53}, 
51--95. Amer. Math. Soc., Providence, RI, 1986.
	
	\bibitem[IL]{Iwaniec-Li-Ortho}
	H.~Iwaniec and X.~Li.
	\newblock The orthogonality of {H}ecke eigenvalues.
	\newblock {\em Compos. Math.}, 143(3):541--565, 2007.
	
	\bibitem[Iwa]{Iwaniec-Spectral-Weyl}
	H.~Iwaniec.
	\newblock The spectral growth of automorphic {$L$}-functions.
	\newblock {\em J. Reine Angew. Math.}, 428:139--159, 1992.
	
	\bibitem[Joh]{Faa-di-Bruno}
	W.~P. Johnson.
	\newblock The curious history of {F}a\`a di {B}runo's formula.
	\newblock {\em Amer. Math. Monthly}, 109\allowbreak(3):\allowbreak217--234, 2002.
	
	\bibitem[Jut]{Jutila-LS}
	M.~Jutila.
	\newblock On spectral large sieve inequalities.
	\newblock {\em Funct. Approx. Comment. Math.}, 28:7--18, 2000.
	
	\bibitem[Kuz]{Kuznetsov}
	N.~V. Kuznetsov.
	\newblock {P}etersson's conjecture for cusp forms of weight zero and {L}innik's
	conjecture. {S}ums of {K}loosterman sums.
	\newblock {\em Math. Sbornik}, 39:299--342, 1981.
	
	\bibitem[Luo1]{Luo-Twisted-LS}
	W.~Luo.
	\newblock The spectral mean value for linear forms in twisted coefficients of
	cusp forms.
	\newblock {\em Acta Arith.}, 70(4):377--391, 1995.
	
	\bibitem[Luo2]{Luo-LS}
	W.~Luo.
	\newblock Spectral mean-value of automorphic {$L$}-functions at special points.
	\newblock  {\em Analytic {N}umber {T}heory, {V}ol.\ 2 ({A}llerton {P}ark, {IL},
		1995)}, {Progr. Math.}, vol. 139,  621--632. Birkh\"auser
	Boston, Boston, MA, 1996.
	
	\bibitem[Qi]{Qi-GL(3)-Special-Points}
	Z.~Qi.
	\newblock The second moment of {$\mathrm{GL}_3 \times \mathrm{GL}_2$}
	{$L$}-functions at special points.
	\newblock {\em Math. Ann.}, 393\allowbreak(1):\allowbreak1429--1457, 2025.
	
	\bibitem[Tit]{Titchmarsh-Riemann}
	E.~C. Titchmarsh.
	\newblock {\em The {T}heory of the {R}iemann {Z}eta-{F}unction}.
	\newblock The Clarendon Press, Oxford University Press, New York, 2nd
	ed., 1986.
	\newblock Edited and with a preface by D. R. Heath-Brown.
	
	\bibitem[You]{Young-GL(3)-Special-Points}
	M.~P. Young.
	\newblock The second moment of {$GL(3)\times GL(2)$} {$L$}-functions at special
	points.
	\newblock {\em Math. Ann.}, 356(3):1005--1028, 2013.
	
\end{thebibliography}

\def\cprime{$'$}

\end{document}